\documentclass{article}
\usepackage{amsmath}
\usepackage{amsfonts}
\usepackage{amsthm}
\usepackage{amssymb}
\usepackage{centernot}
\usepackage{url}

\numberwithin{equation}{section}

\def\congruent{\equiv}

\def\lcm{{\rm lcm}}

\def\ratQ{\mathbb{Q}}
\def\intZ{\mathbb{Z}}
\def\natN{\mathbb{N}}

\newtheorem{Theorem}{Theorem} 
\newtheorem{Lemma}{Lemma}
\newtheorem{Conjecture}{Conjecture}

\newtheorem{Proposition}{Proposition}

\def\Legendre#1#2{\left( \frac{#1}{#2} \right)} 
 
\def\rad{{\rm rad}}

\begin{document}

\begin{center}
{\large\bf 
At most one solution to $a^x + b^y = c^z$ for some ranges of $a$, $b$, $c$  
}
  
\bigskip 

Robert Styer

\bigskip


2020 Subject Class: {11D61}

\end{center}


5 Feb 2024

\bigskip

\begin{abstract}  We consider the number of solutions in positive integers $(x,y,z)$ for the purely exponential Diophantine equation $a^x+b^y =c^z$ (with $\gcd(a,b)=1$).  Apart from a list of known exceptions, a conjecture published in 2016 claims that this equation has at most one solution in positive integers $x$, $y$, and $z$.  We show that this is true for some ranges of $a$, $b$, $c$, for instance, when $1 < a,b < 3600$ and $c<10^{10}$.  The conjecture also holds for small pairs $(a,b)$ independent of $c$, where $2 \le a,b \le 10$ with $\gcd(a,b)=1$.  We show that the Pillai equation $a^x - b^y = r > 0$ has at most one solution (with a known list of exceptions) when $2 \le a,b \le 3600$.  Finally, the primitive case of the Je\'smanowicz conjecture holds when $a \le 10^6$ or when $b \le 10^6$.  This work highlights the power of some ideas of Miyazaki and Pink and the usefulness of a theorem by Scott.  
\end{abstract}


\section{Introduction}

Let $\natN$ be the set of all positive integers. Let $a$, $b$, $c$ be fixed coprime positive integers with $\min(a,b,c)>1$.  We consider the number of solutions $N(a,b,c)$ in positive integers $(x,y,z)$ to the equation
$$a^x + b^y = c^z. \eqno{(1.1)}$$
Mahler \cite{M} used his p-adic analogue of the method of Thue-Siegel to prove that (1.1) has only finitely many solutions $(x, y, z)$. Gelfond \cite{Ge} later made Mahler's result effective. A result of Beukers and Schlickewei \cite{BS} implies the existence of a bound on the number of solutions, independent of $a$, $b$, and $c$. Hirata-Kohno \cite{HK} used \cite{BS} to obtain a bound of $2^{36}$ (reportedly Hirata-Kohno may have later announced a bound of 200, apparently unpublished).
In \cite{ScSt6}, it is shown that if $c$ is odd, then (1.1) has at most two solutions.  Hu and Le \cite{HL2} showed that when $c$ is even and $\max(a,b,c) > 10^{62}$ then (1.1) has at most two solutions.  Miyazaki and Pink \cite{MP} recently proved that (1.1) has at most two solutions in all cases (except the well known instance $(a,b,c)=(3,5,2)$).  

The following conjecture appeared in \cite{ScSt6}:

\begin{Conjecture}  
Let $a$, $b$, $c$ be coprime positive integers greater than one with $a$, $b$, $c$ not perfect powers and $a<b$.  
Then $N(a,b,c) \le 1$, except for 

(i)  $N(2, 2^r-1, 2^r+1)=2$, $(x,y,z)=(1,1,1)$ and $(r+2, 2, 2)$, where $r$ is a positive integer with $r \ge 2$.  

(ii)  $N(2, 3, 11)=2$, $(x,y,z)=(1,2,1)$ and $(3,1,1)$.

(iii)  $N(2, 3, 35)=2$, $(x,y,z)=(3,3,1)$ and $(5,1,1)$.

(iv)  $N(2, 3, 259)=2$, $(x,y,z)=(4,5,1)$ and $(8,1,1)$.

(v)  $N(2, 5, 3)=2$, $(x,y,z)=(1,2,3)$ and $(2,1,2)$.

(vi)  $N(2, 5, 133)=2$, $(x,y,z)=(3,3,1)$ and $(7,1,1)$.

(vii)  $N(2, 7, 3)=2$, $(x,y,z)=(1,1,2)$ and $(5,2,4)$.

(viii)  $N(2, 89, 91)=2$, $(x,y,z)=(1,1,1)$ and $(13,1,2)$.

(ix)  $N(2, 91, 8283)=2$, $(x,y,z)=(1,2,1)$ and $(13,1,1)$.

(x)  $N(3,5,2)=3$, $(x,y,z)=(1,1,3)$, $(1,3,7)$, and $(3,1,5)$.

(xi)  $N(3,10,13)=2$, $(x,y,z)=(1,1,1)$ and $(7,1,3)$.

(xii)  $N(3,13,2)=2$, $(x,y,z)=(1,1,4)$ and $(5,1,8)$.

(xiii)  $N(3, 13, 2200)=2$, $(x,y,z)=(1,3,1)$ and $(7,1,1)$.

\end{Conjecture}

Le, Scott and the author (\cite{LeSt}, \cite{LSS}) deal with the special case where $a$, $b$, and $c$ are primes.  Miyazaki and Pink \cite{MP2} have shown that there are no other double solutions when $c=2$, 6, or a Fermat prime (with partial progress on many other values of $c$).  In \cite{ScSt6}, straightforward computations show that there are no other solutions in the range $a < 2500$, $b < 10000$ with the restriction $a^x, b^y < 10^{30}$.  Benjamin Matschke \cite{Mat} has impressive calculations on the $abc$ conjecture; he found all 432408 solutions with ${\rm rad}(abc)<10^7$.  From his list one can show that Conjecture 1.1 is true for $(a,b,c)$ with ${\rm rad}(abc)<10^7$.  

Here we establish the following: 

\begin{Theorem}  
For given coprime $a$, $b$, and $c$ which are not perfect powers with 
$$1 < a < 3600, 1 < b < 3600, c < 10^{10}$$
or 
$$1 < a \le 130, 1 < b < 10^5, c < 10^{10},$$
or 
$$1 < a \le 100, 1 < b \le 100, c < 10^{1000}, c \equiv 1 \bmod 2,$$ 
the equation $a^x + b^y = c^z$ has at most one solution $(x,y,z)$ in positive integers except for the cases listed in Conjecture 1.1.  
\end{Theorem}

The main purpose of this paper is to highlight how ideas of Miyazaki and Pink \cite{MP} and an improved version of a theorem of Scott \cite{Sc} can be used to prove results like Theorem 1. In later sections we will apply Scott's improved theorem (Theorem 2 below) to show the conjecture holds for small $(a,b)$ independent of $c$, and to obtain new results on the primitive case of the Je\'smanowicz' conjecture and on the Pillai equation $a^x-b^y=r$.  

We begin with some lemmas. Lemmas 1 and 2 deal with $c$ even. Lemmas 3 and 5 deal with $c$ odd.      

Let $\nu_2(n) = t$ where $2^t \parallel n$.  Let $\log_*(n) = \max( 1, \log(n))$. The first lemma, Lemma 3.3 of \cite{MP} (based on a general bound by Bugeaud \cite{Bu}) provides a bound on $z$.

\begin{Lemma}[Miyazaki and Pink]  
Assume $\max(a,b) \ge 9$, $c$ even.  Put $\alpha = \min( \nu_2(a^2 -1 ) -1, \nu_2(b^2-1)-1)$, $\beta=\nu_2(c)$.  Let $(x,y,z)$ be a solution of (1.1) with $z>1$.  Then 
$$ z < \log(a) \log(b) \max( k_1, k_2 \log_*^2(k_3 \log(c))) \eqno{(1.2})$$
where
$$ (k_1,k_2,k_3) = \left( \frac{1803.3 m_2}{\beta}, \frac{ 23.865 m_2}{\beta}, \frac{ 143.75 (m_2+1)}{\beta} \right)$$
when $\alpha = 2$, and when $\alpha \ge 3$,
$$ (k_1,k_2,k_3) = 
\left( 
\frac{2705 m_3}{\alpha \beta}, 
\frac{ 156.39 m_3 \left( 1 + \frac{\log(v_a)}{v_a-1} \right)^2 }{\alpha^3 \beta},
\frac{ 646.9 (m_3+1)}{\alpha^2 \beta} 
\right).
$$
Here $v_a = 3 \alpha \log(2) - \log( 3 \alpha \log(2))$, $m_2 = 1$ when $\min(a,b) > 7$ and $m_2 = \log(8)/\log(\min(a,b))$ when $\min(a,b) \le 7$, and $m_3 = \alpha \log(2) / \log(2^\alpha -1)$.  

$k_1$, $k_2$, $k_3$ decrease when $\alpha \ge 3$ and $\beta$ increase.  
\end{Lemma}

When $c$ is even and $\max(a,b) < 10$, Theorem 7.2 of Bennett and Billerey \cite{BeBi} shows that (1.1) has no double solutions other than those listed in the conjecture. 

Suppose (1.1) has two solutions $(x_1, y_1, z_1)$ and $(x_2, y_2, z_2)$ with $z_1 \le z_2$.  For given $a$, $b$, and even $c$, the proof of Lemma 5.1 of \cite{MP} yields

\begin{Lemma}[Miyazaki and Pink]    
Let $c$ be even, $\alpha$ and $\beta$ defined as in Lemma 1, and $z_1 \le z_2$.  We have 
$$ \beta z_1 - \frac{\log(z_1)}{\log(2)} < \alpha + \frac{1}{\log(2)} \log\left( \frac{\log^2(c)}{\log(a) \log(b)} z_2 \right).  \eqno{(1.3)}$$ 
\end{Lemma} 

The same 2-adic arguments of Miyazaki and Pink apply when $c$ is odd, only now giving us a bound on $\min(x_1, x_2)$ instead of $\min(z_1, z_2)$.   

\begin{Lemma} 
Let $a$ be even.  Let $\nu_2(a) = \gamma$ and $\nu_2(b^2-1)-1 = \delta$.  Suppose (1.1) has two solutions $(x_1, y_1, z_1)$ and $(x_2, y_2, z_2)$ with $x_1 \le x_2$.  Then 
$$ \gamma x_1 < \delta + \frac{1}{\log(2)} \log\left( \frac{\log(c)}{\log(b)} z_1 z_2 \right). \eqno{(1.4)} $$
\end{Lemma}

\begin{proof}  (Following the ideas in the proof of Lemma 5.1 in \cite{MP}.)
From $a^x < c^z$ and $b^y < c^z$ we obtain the bounds $x_1 \le z_1 \log(c)/\log(a)$, $y_1 \le z_1 \log(c)/\log(b)$, $x_2 \le z_2 \log(c)/\log(a)$, $y_2 \le z_2 \log(c)/\log(b)$. 

From (1.1) we have $b^{y_1} \equiv c^{z_1} \bmod a^{x_1}$ and $b^{y_2} \equiv c^{z_2} \bmod a^{x_2}$, so $b^{y_1 z_2} \equiv c^{z_1 z_2} \bmod a^{x_1}$ and $b^{y_2 z_1} \equiv c^{z_2 z_1} \bmod a^{x_2}$.  Thus, $b^{y_1 z_2} \equiv b^{y_2 z_1} \bmod a^{x_1}$ hence $b^{|y_1 z_2 - y_2 z_1|} \equiv 1 \bmod a^{x_1}$.  Thus, $2^{\gamma z_1} | b^{|y_1 z_2 - y_2 z_1|}-1$.  Hu and Le \cite[Lemma 3.3]{HL} show that $y_1 z_2 \ne y_2 z_1$; their result assumes $a^{x_1} > 2$, but (1.4) is clear when $a=2$ and $x_1=1$. Applying well known 2-adic properties, we have $\gamma z_1 \le \nu_2(b^2-1)-1 + \nu_2(|y_1 z_2 - y_2 z_1|)$.  

Now 
$$|y_1 z_2 - y_2 z_1| < \max(y_1 z_2, y_2 z_1) \le \max\left( z_1 \frac{\log(c)}{\log(b)} z_2, z_2 
\frac{\log(c)}{\log(b)} z_1 \right) = \frac{\log(c)}{\log(b)} z_1 z_2.  $$
Thus, 
$$ \gamma x_1 \le \nu_2(b^2-1)-1 + \nu_2(|y_1 z_2 - y_2 z_1|) < \delta + \frac{1}{\log(2)} 
\log\left( \frac{\log(c)}{\log(b)} z_1 z_2 \right).$$
\end{proof} 

Miyazaki and Pink use known results on generalized Fermat equations to eliminate many cases.  The following is based on their Lemma 8.1.

\begin{Lemma}  
Equation (1.1) has no solutions $(x,y,z)$ in the following cases:
\begin{align*}
& x \equiv y \equiv z \equiv 0 \bmod N, N \ge 3, \\
& x \equiv y \equiv 0 \bmod N, z \equiv 0 \bmod 2, N \ge 4, \\
& x \equiv y \equiv 0 \bmod N, z \equiv 0 \bmod 3, N \ge 3, \\ 
& x \equiv 0 \bmod 2, y \equiv 0 \bmod 4, z \ge 4, \\
& x \equiv 0 \bmod 2, y \ge 4, z \equiv 0 \bmod 4, \\
& x \equiv 0 \bmod 2, y \ge 3, z \equiv 0 \bmod 6, \\
& x \equiv 0 \bmod 2, y \equiv 0 \bmod 6, z \ge 3, \\
& x \equiv 0 \bmod 3, y \equiv 0 \bmod 3, z \equiv 0 \bmod N, 3 \le N \le 10^9, \\
& x \equiv 0 \bmod 3, y \equiv 0 \bmod 4, z \equiv 0 \bmod 5, \\
& x \equiv 0 \bmod 2, y \equiv 0 \bmod 3, z \equiv 0 \bmod N, N \in \{ 7,8,9,10,15\}. \\
\end{align*}
\end{Lemma} 

When $c$ is odd, we use a significant improvement of Scott's Theorem 2 in \cite{Sc} to bound the values of $z$. We note that the set of possible $z$ values given here is independent of $c$.   

\begin{Theorem}[Scott]  
Let $R$ be a set of positive rational primes, let $S$ be the set of all integers greater than one all of whose prime divisors are in $R$, and let $T$ be the set of all integers in $S$ divisible by every prime in $R$.  Let $P$ and $Q$ be relatively prime squarefree integers such that $PQ \in T$.  Take $A, B \in S$ such that $AB \in T$, $\gcd(A,B)=1$, and $(AB/P)^{1/2}$ is an integer.  

Then for odd $c$ with $\gcd(c, AB)=1$, suppose 
$$ A + B = c^z  \eqno{(1.5)}$$ 
has a solution $(A, B, z)$.  Then 
$$ z \big\vert \frac{3^{u+v}}{2} h(-P) t_k \eqno{(1.6)}$$
where $t_k = q_k - \Legendre{-P}{q_k}$ for some $k = 1, \dots, n$ with $Q = q_1 q_2 \cdots q_n$ the prime factorization of $Q$, $h(-P)$ is the least $h$ such that $\mathfrak{a}^h$ is principal for each ideal $\mathfrak{a}$ in $\ratQ(\sqrt{-P})$, $u=1$ or $0$ according as $ 3 < P \equiv 3 \bmod 8$ or not, and $v = 1$ or 0 according as $\{ A, B \}$ is or is not $\{ 3^{2N+1} \frac{3^{N-1}-1}{8}, \frac{3^{N+1}-1}{8} \}$ for odd $N>1$.  Here we set the Legendre symbol $\Legendre{-P}{q_k} = 0$ when $q_k=2$. 
\end{Theorem} 

In \cite{Sc} Scott had $\frac{3^{u+v}}{2} h(-P) \lcm(t_1, t_2, \dots, t_n)$ in place of the right side of (1.6) and also handled the case $c=2$; the proof in \cite{Sc} is elementary.  Circa 2005 Scott announced a stronger version of his 1993 theorem; in that announcement the $t_k$ values above can often be reduced by an extra factor of 2, but the statement of this stronger theorem is more complicated and the proof is much more complicated; the slightly weaker theorem given here is sufficient for our purposes.  

\begin{proof} 
Assume that (1.5) has a solution $(A,B,z)$, where we assume for convenience that $A<B$.  From (1.5) we have 
$$ (B-A + 2 \sqrt{-AB} ) (B-A - 2 \sqrt{-AB} ) = c^{2z}, $$
giving the equation in ideals 
$$ [B-A + 2 \sqrt{-AB} ] = \mathfrak{c}^{2z},  \eqno{(1.7)} $$
where $\mathfrak{c}$ is an ideal in $\ratQ(\sqrt{-P})$ such that $\mathfrak{c} \overline{ \mathfrak{c} }   = [c]$ and $\mathfrak{c}$ is not divisible by a principal ideal having a rational integer generator.  We say that a solution $(A,B,z)$ satisfying (1.7) is {\it associated} with the ideal factorization $\mathfrak{c} \overline{\mathfrak{c} }$.  We use the notation of \cite{Sc} except that we replace the $z$ in \cite{Sc} by $w$ to avoid confusion with the $z$ in (1.5).

Let $w$ be the least positive integer such that $\mathfrak{c}^w$ is a principal ideal having a generator with rational integer coefficients. Write 
$$ [a_w + b_w \sqrt{-P}] = \mathfrak{c}^w, [a_i + b_i \sqrt{-P}] =  [(a_w + b_w \sqrt{-P})^{i/w}] = [a_w + b_w \sqrt{-P}]^{i/w} = \mathfrak{c}^i$$
where $i$ can be any positive multiple of $w$.  If $P = 3$, we can take $a_w$, $b_w \in \intZ$, so that $|a_w|$, $|b_w|$ are uniquely determined; if $P=1$, we can take $b_w$ even, so that $|a_w|$, $|b_w|$ are uniquely determined.  

Let $j$ be the least number such that $2 Q \mid b_{j}$.  Then Observation 1 below follows from \cite{Sc} (Theorems 1 and 2 and their proofs; see also \cite{ScSt8} for a somewhat easier and more direct presentation). 

Observation 1 (\cite{Sc}, \cite{ScSt8}):  If (1.5) has a solution $(A,B,z)$ associated with $\mathfrak{c} \overline{ \mathfrak{c}}$ (with $A$ and $B$ satisfying all the restrictions in the statement of this theorem) then it has such a solution with $z = j/2$, where $j$ is defined as above for this $\mathfrak{c} \overline{ \mathfrak{c}}$.  This is the only such solution, except in the special case given by $v=1$ in the statement of the theorem, in which case $z=\frac{3j}{2}$, and there is no third solution associated with this $\mathfrak{c} \overline{ \mathfrak{c}}$.

(Note that although Theorem 1 of \cite{Sc} is given for $c$ prime, the derivation works just as well for $c$ composite if an ideal factorization $\mathfrak{c} \overline{ \mathfrak{c}}$ is specified (this is pointed out in the proof of Theorem 2 of \cite{Sc}).  The case of composite $c$ is handled more directly in \cite{ScSt8}.  Also the case $A=1$ is handled more directly in \cite{ScSt8}.  Note that the definitions of $j$ and $z$ in \cite{ScSt8} are not the same as in \cite{Sc}. \cite{ScSt8} is a revision of \cite{ScSt7}: the proofs of the two main lemmas (Lemmas 1 and 2) in \cite{ScSt8} are quite different from the corresponding proofs in \cite{ScSt7}.)

We see from (1.7) that $A$ and $B$ are completely determined for a given $z$ and a given $\mathfrak{c} \overline{ \mathfrak{c}}$ (even when $P = 1$ or 3), so that, for a given $z$, there is at most one pair $(A,B)$ such that $A+B = c^z$ with $(A,B,z)$ associated with $\mathfrak{c} \overline{ \mathfrak{c}}$, even if we allow $A$ and $B$ to be any coprime positive integers (and generalize the definition of \lq associated' appropriately).  Thus we have 

Observation 2:  The solution $(A,B,j/2)$ referred to in Observation 1 is unique in the sense that there are no other choices of $A$ and $B$ such that $A + B = c^{j/2}$ with $(A,B, j/2)$ associated with $\mathfrak{c} \overline{ \mathfrak{c}}$, even if we allow $A$ and $B$ to be any coprime positive integers.

Now assume that (1.5) has a solution $(A,B,z)$ associated with $\mathfrak{c} \overline{ \mathfrak{c}}$ with $A$ and $B$ satisfying the restrictions of the statement of this theorem, and define $j$ and $w$ as above for this $\mathfrak{c} \overline{ \mathfrak{c}}$.  By Observation 1 we have 
$$ z = 3^v \frac{j}{2},  \eqno{(1.8)}$$
where $v$ is as in the statement of the theorem.  From (1.8) we obtain 
$$ z = 3^v \frac{j}{2} = 3^v \frac{w}{2} \frac{j}{w} \mid \frac{3^{u+v}}{2} h(-P) \frac{j}{w} \eqno{(1.9)}$$

So it suffices to show that for some $k$, $ 1 \le k \le n$, 
$$ \frac{j}{w} \mid t_k. \eqno{(1.10)}$$
(Note that $t_k$ is independent of the choice of $\mathfrak{c} \overline{ \mathfrak{c}}$.) 

For each $k$, $ 1 \le k \le n$, let $g_k$ be the least number such that $q_k \mid b_{w g_k} $.  It is a familiar elementary result on the divisibility properties of the numbers $b_i$ that 
$$ g_k \mid t_k. \eqno{(1.11)} $$ 
If $b_j$ has a primitive divisor, then, by Observation 1 and (1.7), this primitive divisor divides $Q$, so that, by (1.11), we obtain (1.10).  So we can assume $b_j$ has no primitive divisor.  For this case it is helpful to observe that $j$ is the least number such that $Q \mid b_j$ (recall $j$ is defined as the least number such that $2Q \mid b_j$).  To see this, assume there is some $j_0 < j$ such that $Q \mid b_{j_0}$, and choose $j_0$ to be the least such number.  By Lemma 2 of \cite{Sc}, $j_0 \mid i$ for every $i$ such that $Q \mid b_i$.  Then we must have $j = 2 j_0$, so that $a_{j_0}^2 + b_{j_0}^2 P = c^{j_0} = c^{j/2}$ with $(a_{j_0}^2, b_{j_0}^2 P, j/2)$ associated with $\mathfrak{c} \overline{ \mathfrak{c}}$, so that $(a_{j_0}^2, b_{j_0}^2 P, j/2)$ must be the unique solution $(A,B,j/2)$ referred to in Observation 2, so that $a_{j_0}^2 \in S$ and $b_{j_0}^2 P \in S$.  
Since $Q \mid b_{j_0}$, we have $b_{j_0} P \in T$, so, since $\gcd(a_{j_0}, b_{j_0} P) = 1$, we must have $a_{j_0}^2 \not\in S$, giving a contradiction.  So we can assume that $j$ is the least number such that $Q \mid b_j$.  

So now, still assuming $b_j$ has no primitive divisor, choose $j_1$ and $j_2$ as follows: 
let $g_m$ be the greatest of the $g_k$; take $j_2 = w g_m$ (so that $j_2 < j$) and take $j_1 = w g_h$ for some $h$, $1 \le h \le n$ such that $g_h \nmid g_m$ (such $j_1$ is possible since, if $g_h \mid g_m$ for every $h$, $1 \le h \le n$, then $Q \mid b_{w g_m} = b_{j_2}$, contradicting $j_2 < j$).  

If $j/w$ is a prime or prime power, then, since $j_1 \mid j$ and $j_2 \mid j$, we must have $j_1 \mid j_2$, contradicting $g_h \nmid g_m$.  And if $j/w = 2p$ for some odd prime $p$, then the only possibility for $(j_1/w, j_2/w)$ is $(2, p)$, in which case $p = g_m \mid t_m$ so that $2p \mid t_m$ and (1.10) holds with $k=m$.  So we can assume $j/w$ is not a prime or prime power and is also not equal to $2p$ for any odd prime $p$.  

Now from Table 1 of \cite{BHV} we find that the only possible choices for $j/w$ are 12, 18, and 30.  If $j/w = 18$ (respectively, 30), we see that from Table 1 of \cite{BHV} that $b_{9w}$ (respectively, $b_{15w}$) has a primitive divisor which must divide $Q$ (by Lemma 1 of \cite{Sc}, Observation 1, and (1.7)), so we can take this primitive divisor to be $q_k$ for some $1 \le k \le n$, so that 9 (respectively, 15) equals $g_k$ which divides $t_k$, so that $18$ (respectively, 30) divides $t_k$ and (1.10) holds.  

So we are left with $j/w = 12$.  Since $a_w + b_w \sqrt{-P}$ must be an integer in an imaginary quadratic field with $a_w$ and $b_w$ rational integers, 
Table 1 of \cite{BHV} shows that we need to consider only one case: 
$$ a_w + b_w \sqrt{-P} = 1 + \sqrt{-14}.$$
For this case $w=1$ and we find $11 \mid b_3 \mid b_{12}$, so that $11 \mid Q$, and, taking $q_k = 11$, we find $t_k = 12 = j/w$, so that (1.10) holds.  
\end{proof}

\section{Proving Theorem 1}  

We now outline the algorithms used to verify Theorem 1.  Assume we have two solutions to (1.1), $a^{x_1}+b^{y_1} = c^{z_1}$ and $a^{x_2} + b^{y_2} = c^{z_2}$.   

\noindent
Case 1: $c$ is even.  

Choose $a,b < 3600$ odd, not perfect powers, $\gcd(a,b)=1$, and assume $b<a$.  Set $M_c = 10^{10}$.  Suppose $z_1 \le z_2$.  Replacing $c$ by $M_c$ in (1.2), we obtain a bound on $z_2$, call this $M_2$.  Replacing $z_2$ by $M_2$ in (1.3), we use (1.3) to find a bound $M_1$ on $z_1$.  For each $z_1 \le M_1$, we now have bounds 
$$ x_1 \le z_1 \frac{\log(M_c)}{\log(a)}, y_1 \le z_1 \frac{\log(M_c)}{\log(b)}.$$
For given $(x_1,y_1,z_1)$ within these bounds, we first see if Lemma 4 eliminates it. If $z_1=1$, we know $c = a^{x_1} + b^{y_1}$.  If $z_1>1$, we determine whether $a^{x_1} + b^{y_1}$ is a $z_1$ power of an integer.  Almost always, $a^{x_1} + b^{y_1}$ is not a $z_1$ power; if it is, we know $c$. In either case, if there is a solution, we know the value of $c$.

We use (1.2) with this value of $c$ to get a bound $M_{z_2}$ on $z_2$.  For each $z_2< M_{z_2}$, either $c^{z_2}/2 < a^{x_2} < c^{z_2}$ or $c^{z_2}/2 < b^{y_2} < c^{z_2}$.  

Suppose $c^{z_2}/2 < a^{x_2} < c^{z_2}$.  Then we obtain a very tight bound on $x_2$:  
$$z_2 \frac{\log(c)}{\log(a)} - \frac{\log(2)}{\log(a)} < x_2 <  z_2 \frac{\log(c)}{\log(a)}.$$
For any integer $x_2$ in this range, we calculate $c^{z_2} - a^{x_2}$ and see if this is a perfect power of $b$, in which case we have obtained the value of $y_2$ and we have found a double solution.  When $c^{z_2}/2 < b^{y_2} < c^{z_2}$ the argument is similar.  

\noindent
Case 2: $c$ is odd.  

Assume $1 < a < 3600$ is even, $1 < b < 3600$ is odd, $\gcd(a,b) =1$, and neither is a perfect power.  We have four possible parity classes of exponents: $(x,y) \equiv (0,0)$, $(1,0)$, $(0,1)$, and $(1,1) \bmod 2$.  Consider one of these four parity classes. Given $a$ and $b$, we use Theorem 2 to find a list of all possible values for $z_1$ for this parity class.  

Consider $z_1$ from this list of possible $z_1$ values.  We use Theorem 2 to find the list of possible $z_2$ values by taking the union of possible $z$ values over all four parity classes. Assume $x_1 \le x_2$ (in Case 2 we no longer assume $z_1 \le z_2$).  We obtain a bound on $x_1$ by using (1.4) in Lemma 3 with $M_c = 10^{10}$ replacing $c$ and the maximum of the $z_2$ values replacing $z_2$.  Use $y_1 \le z_1 \log(c)/\log(b)$ to get a bound on $y_1$.  Consider each pair $(x_1, y_1)$ within these bounds in the given parity class.  Lemma 4 eliminates a significant number of $(x_1,y_1, z_1)$.  For the remaining cases, if $z_1 = 1$ then we know the value of $c = a^{x_1} + b^{y_1}$.  If $z_1>1$,  we check if $ a^{x_1} + b^{y_1}$ is a $z_1$ power of an integer $c$; almost always it will not be; if it is, we know $c$. In either case, if there is a solution, we know the value of $c$.  

Recall that Theorem 2 gave us a list of possible $z_2$ values.  If $a^{x_2} + b^{y_2} = c^{z_2}$, either $c^{z_2}/2 < a^{x_2} < c^{z_2}$ or $c^{z_2}/2 < b^{y_2} < c^{z_2}$. As above, when $c^{z_2}/2 < a^{x_2} < c^{z_2}$ we obtain very tight lower and upper bounds on $x_2$, and easily check if $c^{z_2} - a^{x_2}$ is a perfect power of $b$ for each $x_2$ within the bounds.  Similarly, when $c^{z_2}/2 < b^{y_2} < c^{z_2}$ we obtain very tight bounds on $y_2$ and easily check if $c^{z_2} - b^{y_2}$ is a perfect power of $a$.  We try each $z_2$; if no $x_2$ or $y_2$ results in a perfect power, then we have verified that there is no second solution.  

We used Sage \cite{Sage} to preprocess the $h(-P)$ values for every $P< 13 \cdot 10^6$, then ran a Python script on a high performance computing cluster for the remaining calculations. Note that $\sqrt{13 \cdot 10^6} = 3605.55$ which is why 3600 is the bound for many of our calculations.

\section{Primitive case of the Je\'smanowicz conjecture}    

The ideas used for the case of $c$ odd can be applied to the primitive case of the Je\'smanowicz conjecture. Le et al \cite{LSS1} summarize a large number of results on the conjecture.  To the best of this author's knowledge, no one has considered explicit lower bounds for possible solutions to the primitive case of the Je\'smanowicz conjecture.

\begin{Theorem}  
Consider a primitive Pythagorean triple $(a,b,c)$.  If $a \le 10^6$ or if $b \le 10^6$, then the primitive case of the Je\'smanowicz conjecture holds, that is, the only solution to $a^x + b^y = c^z$ is $(x,y,z)=(2,2,2)$.
\end{Theorem}

\begin{proof} 
Without loss of generality, let $a= f^2 - g^2$, $b = 2fg$, and $c = f^2 + g^2$ for positive relatively prime integers $f$ and $g$ of opposite parities.  

We can use a theorem of Han and Yuan \cite{HY} to eliminate almost half the possible $(f,g)$ cases. They showed that if $fg \equiv 2 \bmod 4$ and if $f+g$ has a prime divisor $p \not\equiv 1 \bmod 16$, then the primitive case of the Je\'smanowicz conjecture holds. 

We first consider the cases with $g<f \le 1000$.  We use Scott's Theorem 2 above to obtain a list of possible $z_2$ values.  For each $z_2$, we have the bounds $ x_2 \le z_2 \log(c)/\log(a)$ and $y_2 \le z_2 \log(c)/\log(b)$; as in Case 2 ($c$ odd) of the previous section, we obtain very tight bounds on possible values of either $x_2$ or $y_2$ and verify that there is no second solution. 

Now we consider the case $a \le 10^6$ where $a = f^2 - g^2$ with $f>1000$.  Demjanenko \cite{D} showed that there are no solutions when $g = f-1$, hence $g \le f-3$.  Since $a = f^2 - g^2 \le 10^6$, $\sqrt{f^2 - 10^6} \le g \le f-3$ gives us good bounds on the value of $g$ for $f>1001$, and also implies $f \le 166668$.  As before, Theorem 2 provides a list of possible $z_2$ values.  For each $z_2$, we have the bounds $ x_2 \le z_2 \log(c)/\log(a)$ and $y_2 \le z_2 \log(c)/\log(b)$; as in Case 2 ($c$ odd) of the previous section, we obtain very tight bounds on possible values of either $x_2$ or $y_2$ and verify that there is no second solution. 

Lastly, we consider the case $b \le 10^6$ where $b = 2 fg$ with $f>1000$.  Lu \cite{Lu} showed that the primitive case of the conjecture holds when $g=1$ and Terai \cite{T} showed it holds when $g=2$, so we can assume $g \ge 3$.  If $b = 2fg \le 10^6$ then $ 3 \le g \le  500000/f$ and $f \le  500000/3$ give us reasonable bounds on the possible $(f,g)$ pairs.  Once again, Theorem 2 provides a list of possible $z_2$ values.  For each $z_2$, we have the bounds $ x_2 \le z_2 \log(c)/\log(a)$ and $y_2 \le z_2 \log(c)/\log(b)$; as in Case 2 ($c$ odd) of the previous section, we obtain very tight bounds on possible values of either $x_2$ or $y_2$ and verify that there is no second solution. 
\end{proof}

\section{Small values of $a$ and $b$}  

When $c$ is even and $\max(a,b) < 10$, Theorem 7.2 of Bennett and Billerey \cite{BeBi} shows that (1.1) has no double solutions other than those listed in the conjecture. (They allow $\gcd(a,b)>1$ but we will require $\gcd(a,b)=1$.) We can apply Scott's Theorem 2 above to get a  similar result when $c$ is odd.

\begin{Theorem} 
Consider $2 \le a,b \le 10$ with $\gcd(a,b)=1$.  For each such pair $\{a,b\}$, the equation $a^x + b^y = c^z$ has at most one solution $(x,y,z)$ except for $(a,b,c)$ listed in Conjecture 1.   
\end{Theorem}

\begin{proof}
As noted above, when $c$ is even, \cite{BeBi} shows the desired result.  When $c$ is odd, we can apply Theorem 2.  Suppose we have two solutions $a^{x_1} + b^{y_1} = c^{z_1}$ and $a^{x_2} + b^{y_2} = c^{z_2}$ with $z_1 \le z_2$.  

When applying Theorem 2, the $z$ values depend on the $P$ and $Q$ which for $a^x + b^y = c^z$ depend on the parities of $x$ and $y$.  Here is a table listing, for a given $\{a,b\}$, the set of $z$ values possible for each of the four parity classes of $(x,y)$.  

\begin{tabular}{l|l|l|l|l}
\{a, b\} & (0, 0) & (0, 1) & (1, 0) & (1, 1) \\
\hline 
\{2, 3\} &  \{1, 2\} & \{1\} & \{1\} & \{1\} \\
\{2, 5\} & \{1, 2\} & \{1, 2\} & \{1, 3\} & \{1\}  \\
\{2, 7\} & \{1, 2, 4\} & \{1\} & \{1, 2, 4\} & \{1, 2\} \\
\{6, 5\} & \{1, 2\} & \{1, 2\} & \{1, 2, 4\} & \{1\}  \\
\{6, 7\} & \{1, 2, 4\}  & \{1, 2\} & \{1, 2, 3, 6\} & \{1\}  \\
\{10, 3\} & \{1, 2\} & \{1, 3\} & \{1, 2, 4\} & \{1\}  \\
\{10, 7\} & \{1, 2, 4\} & \{1, 3\} & \{1, 2, 3, 6\} & \{1\}  \\
\end{tabular} 

Note that the $z$ values are divisible by no primes other than 2 or 3. In the next section we investigate the Pillai equation and show that for these $(a,b)$ pairs (and many others), we have no double solutions with $z_1=z_2$ except those given in Conjecture 1.  Thus, $z_2>1$ so either $2 \mid z_2$ or $3 \mid z_2$.  

Suppose $2 \mid z_2$.  The $(0,0)$ parity class leads to a Pythagorean triple and can be easily handled.  (Such an analysis completes the $(2,3)$ case.) We will handle the $(2,7)$ case separately below.  For the other pairs we have $(x_2, y_2) \equiv (0,1) \bmod 2$ or $(1,0) \bmod 2$.  

If $x_2$ is even then $y_2$ is odd.  We obtain $b^{y_2} = c^{z_2} - a^{x_2} = (c^{z_2/2} - a^{x_2/2})(c^{z_2/2} + a^{x_2/2})$.  Since $b$ is an odd prime, one can derive the equation $1 = b^{y_2} - 2 a^{x_2/2}$. If $x_2 \ge 4$ then we can view this equation modulo 8 to obtain $y_2$ even, a contradiction.

If $y_2$ is even then $x_2$ is odd (and here we only need to consider $a=6$ or 10). In the same way, we can derive two possible equations:  $1 = a^{x_2}/4 - b^{y_2/2}$ or $ 2^{x_2-2} + b^{y_2/2} = A^{x_2}$ where $a = 2 A$.  Elementary considerations lead to the exceptional cases of the theorem and eliminate all other possibilities. (Side Remark: When we view these equations modulo 4 or modulo 8, we need $x_2 \ge 3$ or 4.  Considering smaller $x_2$ leads to the interesting equations $6^3 + 5^4 = 29^2$ and $10^2 + 3^5 = 7^3$.)  

Suppose $3 \mid z_2$.  For $a=2$, $b=5$, the parity class has $x_2$ odd and $y_2$ even, say $x_2 = 2u+1$ and $y_2 = 2v$.  Now $2^{2u+1} + 5^{2v} = ( 5^v + 2^u \sqrt{-2})(5^v - 2^u \sqrt{-2})$. Since this field is a principal ideal domain, $c = (r+s \sqrt{-2})(r-s \sqrt{-2})$; analyzing the quadratic integer factorizations of $2^{2u+1} + 5^{2v} = c^3$ we can reach a contradiction.  For $(6,7)$, $(10, 3)$, and $(10,7)$, we can use similar ideas, or we can show that one of $x_2$ or $y_2$ is divisible by 3, so we get a power equals a difference of cubes, and with further elementary analysis we show this is not possible.  
\end{proof}     
   
We treat the $(2,7)$ case separately since it is the only case associated to the parity class $(x,y) \equiv (1,1) \bmod 2$ with $z_2>1$, plus this case exemplifies the type of elementary steps used to prove the above cases.       

\begin{Lemma} 
For a given positive integer $c$, there is at most one solution in positive integers $(x,y,z)$ to the equation
$$ 7^x + 2^y = c^z, \eqno{(4.1)}$$
except when $c=3$ which gives the two solutions $(x,y,z) = (1,1,2)$ and $(2,5,4)$.  ($c=9$ gives the equivalent solutions.)
\end{Lemma} 

\begin{proof} 
Assume (1) has two solutions as follows:
$$ 7^{x_1} + 2^{y_1} = c^{z_1}$$
and 
$$7^{x_2} + 2^{y_2} =  c^{z_2}.$$ 
By \cite{ScSt8} there is no third solution.  We cannot have $z_1 = z_2$, since then there exists a nonzero integer $d$ such that the equation $7^x - 2^y = d$ has two solutions $(x,y)$; $d<0$ is impossible by Theorem 6 of \cite{Sc}, and $d>0$ is impossible by Theorem 3 of \cite{Sc} since the parity of $y$ is determined by  consideration modulo 3.  

So by Theorem 2, we find that $z_2$ must be a power of 2.  So consideration modulo 3 gives
$$ 3 \mid c, 2 \nmid y_1 y_2. \eqno{(4.2)}$$

For any solution in which both $x$ and $z$ are even, we can subtract $7^x$ from both sides of (1) and factor the difference of squares to obtain $7^{x/2} = 2^{y-2} - 1$, which gives the exceptional case in the lemma.  

Removing this exceptional case from consideration, we find that we must have $x_2$ odd.  Recalling (2) and using Theorem 2 of this paper, we obtain
$$ 2 \nmid x_2 y_1 y_2, z_1 = 1, z_2 = 2. \eqno{(4.3)}$$

We now obtain results using several moduli:  

Modulo 8:  $y_2=1$.

Modulo 9:  $x_2 \equiv 1 \bmod 6$.

Modulo 7:  $y_1 \equiv 5 \bmod 6$. 

Modulo 13:  $x_2 \equiv 1 \bmod 12$, $2 \mid x_1$, $3 \nmid x_1$.

Modulo 73:  $x_2 \equiv 1 \bmod 24$, $c^{z_1} \equiv \pm 3 \bmod 73$, $2^{y_1} \equiv 4$, $32$, $37 \bmod 73$, $7^{x_1} \equiv -1$, $-7$, $-29$, $33$, $-34$, or $-35 \bmod 73$.  

There are no positive integers $n$ for which $7^n \equiv -29$, $33$, $-34$, or $-35 \bmod 73$. $7^{x_1} \equiv -1 \bmod 73$ requires $3 \mid x_1$, contradicting modulo 13.  $7^{x_1} \equiv -7 \bmod 73$ requires $2 \nmid x_1$, again contradicting modulo 13.      
\end{proof}
 
Remark:  When $c$ is odd, we could use the Bennett and Billerey tables  2-3-5-7 and 2-3-$p$ (with prime $p<100$) referenced in \cite{BeBi} to find more pairs $(a,b)$ satisfying Conjecture 1 independent of $c$.   These tables list every possible solution to the $S$-unit equations $A+B = C^2$ and $A+B = C^3$ where $\rad(AB) \mid 2 \cdot 3 \cdot 5 \cdot 7$ or $\rad(AB) \mid 2 \cdot 3 \cdot p$ for a prime $p<100$.  

For any pair $(a,b)$ with $a$ even and $b$ odd, $\rad(ab) \mid 2 \cdot 3 \cdot 5 \cdot 7$ or $\rad(ab) \mid 2 \cdot 3 \cdot p$ for a prime $p<100$,  $a$ and $b$ not perfect powers, $\gcd(a,b)=1$, and $2 \le a,b < 3600$, we use Theorem 2 to check if the possible $z$ values are only divisible by 2 or 3. There are 856 such pairs. By Theorem 5 below, if there is a double solution to (1.1), then we cannot have $z_1=z_2$;  either $2 \mid z_2$ or $3 \mid z_2$ so any potential double solution would appear in the Bennett-Billerey tables referenced in \cite{BeBi}.  For each such solution in the table we would know $z_2$.  If there is a solution $a^{x_2}+b^{y_2} = c^{z_2}$ with 2 or $3 \mid z_2$, quick calculations show that no $z_1<z_2$ gives a solution to $a^{x_1} + b^{y_1} = c^{z_1}$ other than the listed exceptions in Conjecture 1.  So we obtain 856 pairs $(a,b)$ that satisfy Conjecture 1 independent of $c$. Here are the pairs with $a$ and $b$ less than or equal to 20:  $(a,b) = (2, 3)$, (2, 5), (2, 7), (2, 15), (2, 17), (6, 5), (6, 7), (6, 17), (10, 3), (10, 7), (12, 5), (12, 7), (12, 11), (12, 13), (12, 17), (14, 3), (14, 5), (14, 15), (18, 5), (18, 7), (18, 17), (20, 3), or (20, 7).

\section{Pillai equation}  

In the previous section we postponed dealing with the case $z_1=z_2$.  Here we consider $z_1=z_2$.  If $a^{x_1} + b^{y_1} = a^{x_2}  + b^{y_2}$ then we obtain either $a^{x_1} - b^{y_2} = a^{x_2} - b^{y_1} = r$ or $b^{y_1}-a^{x_2} = b^{y_2} - a^{x_1} = r$ for some positive integer $r<c^{z_1}$. In this section we handle these Pillai equations.  Bennett \cite{Be} conjectures that $a^x - b^y = r$ at most one solution in positive integers $(x,y)$ except for an explicit list of cases with two solutions.  In the previous section we only needed to consider $2 \le a,b \le 10$ but here we will show that Bennett's conjecture holds when $\gcd(a,b)=1$ for a much larger range of values of $a$ and $b$.

\begin{Theorem} 
Let $2 \le a \le 3600$, $2 \le b \le 3600$, $\gcd(a,b)=1$, and $r$ a positive integer. Then $a^x - b^y = r$ has at most one solution in positive integers $x$ and $y$, except when $(a,b,r) = (3, 2,1)$, $(2,3,5)$, $(2,3,13)$, $(2,5,3)$, $(13,3,10)$, $(91, 2, 89)$.   
\end{Theorem}

Before we prove this, we state a sharpened version of Bennett's Theorem 1.3 \cite{Be} for the case $\gcd(a,b)=1$, given as Lemma 6 below.  To prove this lemma, we use three propositions, which closely parallel the treatment in \cite{Be}, but give sharper results.  We prove these propositions for the more general equation 
$$ (-1)^u a^x + (-1)^v b^y = r, u, v \in \{0,1\}  \eqno{(5.1)}  $$ 
but require $\gcd(a,b)=1$ (whereas Bennett's lemmas require $u=0$ and $v=1$ but allow $\gcd(a,b)>1$).  These three propositions, found in \cite{ScSt2b}, improve, simplify, and sharpen the corresponding results in \cite{ScSt2} (the paper \cite{ScSt2b} is an updated version of \cite{ScSt2}).

\begin{Proposition}  
Let $a>1$ and $b>1$ be relatively prime integers.  For $1 \le i \le t$, let $p_i$ be one of the $t$ distinct prime divisors of $a$.  Let $p_i^{g_i} || b^{n_i} \pm 1$, where $n_i$ is the least number such that $p_i | b^{n_i} \pm 1$  (when $p_i =2$ we choose the sign to maximize $g_i$).  

Write
$$S = \sum_i g_i \log(p_i)/ \log(a).$$

Then, if 
$$a^x | b^y \pm 1, \eqno{(5.2)}$$
where the $\pm$ sign is independent of the above, we must have 
$$a^{x-S} | y.$$
\end{Proposition}

\begin{proof}  Let $a = \prod_i p_i^{\alpha_i}$.  If (5.2) holds, then for each $i$, $p_i^{x \alpha_i} | b^y \pm 1$, so that $p_i^{x \alpha_i -g_i} | y$ (in the case $p_i=2$, $\alpha_i=1$, $ 2 \nmid y$ we may have $x \alpha_i < g_i$, but then $y/p_i^{x \alpha_i - g_i}$ is an integer).  Thus, $y$ is divisible by 
$$\prod_i p_i^{x \alpha_i - g_i} = a^{x-S}.$$
\end{proof}

\begin{Proposition}  
Let $a$ and $b$ be relatively prime positive integers with $a>2$, $b >1$, and $(a,b) \ne (3,2)$. Then, in the notation of Proposition 1, 
$$S < \frac{ a \log(b)}{2 \log(a)}.$$
\end{Proposition}

\begin{proof}  
We assume $a>2$ and $(a,b) \ne (3,2)$.  Then if $a$ is odd, $\prod_i p_i^{g_i} \le b^{\phi(a)/2} + 1 \le b^{(a-1)/2} + 1 < b^{a/2}$, verifying Proposition 2 when $a$ is odd.  If $a>4$ is even, then $\prod_i p_i^{g_i} \le b^{\phi(a/2)} < b^{a/2}$ verifying the proposition in this case also.  Finally, when $a=4$, define $g$ so that $2^g || b \pm 1$, where the sign is chosen to maximize $g$.  Then the proposition holds unless $\frac{g \log(2)}{ \log(4) } \ge \frac{4 \log(b)}{ 2 \log(4)}$, that is, unless $2^g \ge b^2$, which is impossible.  
\end{proof}

\begin{Proposition}  
Let $a>2$, $b>1$, and $c>0$ be integers with $(a,b)=1$.  If (5.1) has two solutions $(x_1, y_1)$ and $(x_2, y_2)$, with $x_1 \le x_2$ and $y_1 \le y_2$, and if further $a^{x_1} > c/2$, then
$$x_1 < S + k,$$
where $S$ is defined as in Proposition 1, and $k = \frac{ 8.1 + \log \log(a) }{ \log(a)}$ when $a< 5346$ and $k= 1.19408$ otherwise. 
\end{Proposition}

\begin{proof} 
When $(a,b) = (3,2)$, all cases in which (5.1) has more than one solution are given in \cite{Pi} and the Corollary to Theorem 2 of \cite{ScSt1}; when $(a,b)=(5,2)$, the elementary methods of \cite{Pi} along with the Corollary to Theorem 2 of \cite{ScSt1} suffice to give all cases in which (5.1) has more than one solution.  The proposition holds in all these cases, so we assume from here on that $(a,b) \ne (3,2)$ or $(5,2)$.  

Following closely the method of proof in Bennett's Proposition 4.4 \cite{Be}, assume there are two solutions to (5.1) with $a^{x_1} > r/2$, $y_2 \ge y_1$, and $x_2 \ge x_1 = S +k_1$ with $k_1 \ge k$, where $k$ is defined for each $a$ as in the formulation of this proposition.  
If $y_1 = y_2$, then, using equation (6) of \cite{ScSt2} with the roles of $a$ and $b$ reversed, we see that $x_1 = 1$; so we can take $y_1 < y_2$.  
From the equation 
$$a^{x_1} ( a^{x_2 - x_1} \pm 1 ) =  b^{y_1} ( b^{y_2 - y_1} \pm 1)$$
it follows that 
$$b^{y_2 - y_1} \congruent \pm 1 \bmod a^{x_1}$$   
and so Proposition 1 implies that $y_2 - y_1 \ge a^{x_1 - S}$.  Thus, 
$$y_2 > a^{k_1}. \eqno{(5.3)}$$
On the other hand, $r < 2 a^{x_1}$, so 
$$\log(r) < x_1 \log(a) + \log(2) = (S+ k_1) \log(a) +\log(2).$$
So now we have 
$$ \frac{ y_2  \log(b)}{ \log(r)} > \frac{ a^{k_1} \log(b) }{ (S+k_1) \log(a) + \log(2)} .$$
From Proposition 2 we have 
$$S < \frac{a \log(b)}{ 2 \log(a)} $$
and so 
$$ \frac{y_2 \log(b)}{ \log(r)} > \frac{ a^{k_1} }{ \left( \frac{ a }{ 2 \log(a)} + \frac{ k_1 }{ \log(b) } \right) \log(a) + \frac{\log(2)}{ \log(b)} } > 10.519 $$
where the second inequality follows from $k_1 \ge k$, $a \ge 3$, and $b \ge 2$.  Let 
$$G = \max \left\{ \frac{x_2}{ \log(b)} , \frac{y_2}{ \log(a)} \right\}.$$
Then we have 
$$ \frac{ G }{ 5.2595} \ge \frac{y_2 }{ 5.2595 \log(a)} >  \frac{ 2 \log(r) }{ \log(a) \log(b)} \eqno{(5.4)}.$$
Now let $\Lambda = | x_2 \log(a) - y_2 \log(b)|$.  Applying a theorem of Mignotte as given in Section 3 of \cite{Be}, and using in Mignotte's formula the parameters chosen by Bennett in the proof of Proposition 4.4 of \cite{Be} (recall $(a,b) \ne (3,2)$ or $(5,2)$), we see that we must have either
$$\log(G) \le 8.1 \eqno{(5.5)}$$
or
$$\log(\Lambda) > - 24.2 \left( \log(G) + 2.4 \right)^2 \log(a) \log(b). \eqno{(5.6)}$$
First assume $r>1$.  Assume (5.6) holds.  Then, in the same way we derived equation (11) in \cite{ScSt2} in the proof of Theorem 2 in \cite{ScSt2} (here $c_1 = r$),  we obtain 
$$G < 2  \frac{ \log(r)}{ \log(a) \log(b) } + 24.2 ( \log(G) + 2.4 )^2. \eqno{(5.7)}$$
Using (5.4) we obtain 
$$G < 29.8815 (\log(G) + 2.4)^2,$$
which implies $\log(G) < 8.1$.  So, no matter which of (5.5) or (5.6) holds, we have from (5.3) 
$$e^{8.1} \ge G \ge  \frac{y_2 }{ \log(a)} >  \frac{a^{k_1} }{ \log(a)},$$
which is impossible since $k_1 \ge k$.  

Now assume $r=1$, so that $\Lambda < \log(2)$.  Proceeding as with $r>1$, it is easily seen we can replace (5.7) by 
$$G <  \frac{ \log(2) }{ \log(a) \log(b) } + 24.2 ( \log(G) + 2.4 )^2. \eqno{(5.8)}$$
From (5.8) we again derive 
$$ \frac{a^{k_1} }{ \log(a)} < e^{8.1},$$
impossible since $k_1 \ge k$.  
\end{proof}

We are now ready to prove

\begin{Lemma}  
If $a$, $b$ and $r$ are coprime positive integers with $a$, $b \ge 2$ and
$r > b^{2a \log(a)}$ (or, when $a$ is prime and $(a,b,r) \ne (2,3,13)$, we can take $r > b^a$),
then the equation $a^x - b^y = r$ has at most one solution in positive integers $x$ and $y$.
\end{Lemma}

\begin{proof}
Let $a$, $b$, and $r$ be coprime positive integers with $a, b \ge 2$ and $r > b^{2a \log(a)}$, and assume that the equation $a^x - b^y = r$ has two positive integer solutions $(x_1, y_1)$ and $(x_2,y_2)$ with $x_1 < x_2$.  Since $a^{x_1}>r$ we have 
$$ x_1 > 2 a \log(b). $$
Let $S$ be as in Proposition 1.  Proposition 2 gives $S < a \log(b) / (2 \log(a))$, so that
$$ x_1 - S > \left( 2 - \frac{1}{2 \log(a)} \right)  a  \log(b). \eqno{(5.9)}$$
On the other hand, Proposition 3 gives 
$$ x_1 - S < \max\left( \frac{8.1 + \log\log(a)}{\log(a)} , 1.19408 \right). \eqno{(5.10)} $$
We can assume $a \ge 6$, since the only cases of more than one solution to the equation $a^x-b^y=r$ for $a \le 5$ have been shown to be those given in Conjecture 1, all of which satisfy $r< b^{2a \log(a)}$ ($a=2$ or $4$ is handled in \cite{Sc}; $a=3$ or $5$ is handled in \cite{Be}).  So, taking $a \ge 6$, we see that the right side of (5.9) is minimal for $(a,b)=(7,2)$ and the right side of (5.10) is maximal for $a=6$.  Using these values of $a$ and $b$ we find from (5.9) that $x_1 - S > 8.4$, while from (5.10) we find that $x_1-S < 4.9$.  This contradiction proves the lemma for all $(a,b,r)$ with $r > b^{2 a \log(a)}$.  

Proving the lemma for prime $a$ with $r > b^a$ is already done in \cite{Be}, except that the exceptional case $(a,b,r)=(2,3,13)$ is not mentioned in the statement of Theorem 1.3 of \cite{Be}, since the case $a=2$ has already been handled in Proposition 2.1 of \cite{Be}.   
\end{proof}

\begin{proof}[Proof of Theorem 5]
Suppose $a^x - b^y = r$ has a solution $(x,y)$.  Bennett \cite{Be} uses a result of Mignotte on linear forms in logarithms to obtain an inequality from which  \cite{ScSt2} derives the relation  
$$  \frac{x}{\log(b)} < 2 \frac{\log(r)}{\log(a) \log(b)} + 22.997 \left( \log\left( \frac{x}{\log(b)} \right) + 2.405 \right)^2.$$
By Lemma 6 we can replace $r$ by $b^{2a \log(a)}$.  Thus, 
$$  \frac{x}{\log(b)} <  4 a + 22.997 \left( \log\left( \frac{x}{\log(b)} \right) + 2.405 \right)^2. \eqno{(5.11)} $$ 
If $a^x - b^y = r$ has two solutions $a^{x_1} - b^{y_1} = a^{x_2} - b^{y_2} = r >0$ with $x_2 > x_1 \ge 1$, then from (5.11) we can get a bound on $x_2$ in terms of $a$ and $b$.  For instance, if $a$, $b \le 3600$, then $x_2 < 1194836$.  

For a given $a$, $b < 3600$, we use a technique often called \lq bootstrapping' to show that $x_2>1194836$ and hence $a^x - b^y = r$ cannot have two solutions.  Whenever this technique shows that $x_2>1194836$ then we can conclude that there cannot be a second solution for that $a$, $b$.  See \cite{GLS} for an early application of bootstrapping.  

We rearrange $a^{x_1} - b^{y_1} = a^{x_2} - b^{y_2}$ to obtain the equation
$$ a^{x_1} (a^{x_2-x_1} - 1) = b^{y_1} (b^{y_2-y_1} - 1). \eqno{(5.12)}$$ 

We first find lower bounds for $x_1$ and $y_1$.  Bennett \cite{Be} has shown that if $r \le 100$ then there are no double solutions other than the known exceptions listed in Theorem 3.  So we may assume $r \ge 101$.  Thus, $a^{x_1} - b \ge a^{x_1} - b^{y_1} = r \ge 101$ so $x_1 \ge \lceil\log(b+101)/\log(a)\rceil$.  We can also consider the powers of 2 dividing each side of (5.2).  Suppose $a$ is even and $b$ is odd; let $2^\alpha \parallel a$ and $2^\beta \parallel b-1$.  Then $x_1 \ge \lceil\beta/\alpha\rceil$. Suppose $a$ is odd and $b$ is even; let $2^\alpha \parallel a-1$ and $2^\beta \parallel b$.  Then $y_1 \ge \lceil\alpha/\beta\rceil$.  (Similar bounds can be derived using primes other than 2, which helped in a couple of instances.) If $a$ is even and $b$ is odd, We set $X_1 = \max(\lceil\log(b+101)/\log(a)\rceil,\lceil\beta/\alpha\rceil)$ and $Y_1=1$; if $a$ is odd and $b$ is even, $X_1 = \lceil\log(b+101)/\log(a)\rceil$ and $Y_1 = \lceil\alpha/\beta\rceil$.  We have $x_1 \ge X_1 \ge 1$ and $y_1 \ge Y_1 \ge 1$.

The goal of bootstrapping is to use factors on the left side of (5.12) to get a lower bound on $y_2-y_1$, then we use factors on the right side to get a lower bound on $x_2-x_1$, then use the new factors on the left side to get a better lower bound on $y_2-y_1$, etc., alternating until we achieve the desired lower bound on $x_2-x_1$.      

Fix a bound $B$ (in our case $B = 1194836$).  We will use bootstrapping to find a value $d_x \mid x_2-x_1$ with $d_x>B$.    

Define the multiplicative order $M(m,j) = k$ to mean that $k$ is the least positive integer such that $j^k \equiv 1 \bmod m$.  Recall we have $x_1 \ge X_1$ and $y_1 \ge Y_1$.  Let $d_x = M(b^{Y_1}, a) \mid x_2-x_1$ and let $d_y = M(a^{X_1},b) \mid y_2-y_1$.  

For positive integers $u$ and $v$, define $R(u,v) = r$ where $r$ is the maximal positive integer such that $u = rs$ with $\gcd(r,v)=1$.  Since $a^{X_1} (a^{d_x} -1)$ divides the left side of (5.12), we see that $d_y' = M( a^{X_1} R(a^{d_x} -1, b), b) \mid y_2-y_1$.  Similarly, since $b^{Y_1} (b^{d_y} -1)$ divides the right side of (5.12), we see that $d_x' = M(b^{Y_1} R(b^{d_y} -1,a)), a) \mid x_2-x_1$.  Now we alternate finding new factors on each side and use these factors to find lower bounds on the exponents of the opposite side, until we get $x_2-x_1 > B$.   

A potential difficulty arises: calculating $M(m,j)$ requires explicitly factoring $m$ into primes (and explicitly factoring $p-1$ for each of these prime factors, though this was never a problem). In general the computer cannot fully factor $a^{X_1} (a^{d_x} -1)$ or $b^{Y_1} (b^{d_y} -1)$.  But often we can find some factors $f \mid a^{X_1} (a^{d_x} -1)$ for which we can calculate $d_y' = M(f, b) \mid y_2-y_1$.  Similarly, we often can find some factors $f \mid b^{Y_1} (b^{d_y} -1)$ for which we can calculate $d_x' = M(f,a) \mid x_2-x_1$.  As long as $\lcm(d_x', d_x) > d_x$ or $\lcm(d_y', d_y) > d_y$, we are making progress and can continue the bootstrap.  If we cannot find any suitable factors $f$ which increase the lower bounds on $x_2-x_1$ or $y_2-y_1$, then the bootstrapping fails.  If the bootstrapping succeeds in finding an iterate $d_x \mid x_2-x_1$ with $d_x > B$, then we can conclude that $x_2 >  x_2-x_1 \ge d_x > B = 1194836$ so (5.12) has no solution, hence $a^{x} - b^{y} = r$ cannot have two solutions.  

In one case, using $x_1 \ge X_1$ and $y_1 \ge Y_1$ did not result in bootstrapping success.  To handle such cases, we find $X_0 > X_1$ or $Y_0>Y_1$ which allows bootstrapping to work.  First consider $X_0>X_1$; assume that with this $X_0>X_1$ and $Y_0 = Y_1$ we can use bootstrapping to show that $d_x>B$.  The successful bootstrapping shows that (5.12) has no solutions for pairs $(x_1,y_1)$ for which $x_1 \ge X_0$.  Similarly we consider $Y_0>Y_1$ and assume that with $X_0 = X_1$ and $Y_0 > Y_1$ that we can use bootstrapping to show that $d_x>B$ so that (5.12) has no solutions for pairs $(x_1,y_1)$ with $y_1 \ge Y_0$.  We need to deal with the remaining finite number of pairs $(x_1, y_1)$ with $X_0 > x_1 \ge X_1$ and $Y_0 > y_1 \ge Y_1$.  Fix any such pair $(x_1, y_1)$ and consider any value $x_2$ with $B \ge x_2 > x_1$.  For the given $(x_1, y_1, x_2)$ we solve (5.12) for $y_2$ and determine if $y_2$ is integral or not.  If $y_2$ is integral, we have found a double solution to $a^x-b^y=r$, otherwise (5.12) does not have an integral solution for this $(x_1, y_1, x_2)$.  Thus, a finite number of calculations finds all double solutions if any exists.   
\end{proof}

We give the one example for which we needed to take $X_0 > X_1$ and $Y_0>Y_1$.  Let $a=2661 = 3 \cdot 887$ and $b=20 = 2^2 \cdot 5$.  Then $M(b,a) = 1$ and $M(a,b) = 886$.  Now $a^1 -1 = 2^2 \cdot 5 \cdot 7 \cdot 19$ so $X_1 = 1$, and $b^{886} - 1 = 3 \cdot 7 \cdot 19 \cdot 887 \cdot C_{1148}$ so $Y_1 = 1$ (where $C_{1148}$ is a 1148 digit composite that I was unable to factor).  Now $R(b(b^{886}-1),a)= b \cdot 7 \cdot 19 \cdot C_{1148}$ so the only useful factor we find is $f= 7 \cdot 19 b$.  But $M(f,a)=1$. Similarly, $f = R(a(a-1),b) = a \cdot 7 \cdot 19$ and $M(f,b) = 886$ so again we make no progress; our bootstrapping has failed.  

So we now assume $x_1 \ge X_0 = 2$. Then $M(a^2,b) = 2357646 = 2 \cdot 3 \cdot 443 \cdot 887$.  We use the Maple\textsuperscript{TM} command ifactors/easy to find some medium-sized factors $ 209983 \cdot 688423 \mid b^{2 \cdot 3 \cdot 443} -1$.  Now $M(209983 \cdot 688423, a) = 2589778>B$.  So (5.12) has no solutions if $x_1 \ge X_0 = 2$.  Now assume $y_1 \ge Y_0 =2$.  Now $M(b^2,a) = 20$ and $11^3 \mid a^{20} -1$ and $M(11^3, b) = 605$.  We again use Maple to find factors $150041 \cdot 2209901 \mid b^{605}-1$; $M(150041 \cdot 2209901, a) = 753575900 > 1194836$ so (5.12) has no solutions if $y_1 \ge Y_0 = 2$.  

Thus, the only remaining case we need to consider is $(x_1, y_1) = (1,1)$. For each value of $x_2$ with $1194836 \ge x_2 > x_1 = 1$, we can show that this instance of (5.12)
$$ a^1 (a^{x_2-1} -1) = b^1 (b^{y_2-1} - 1)$$
has no integral solution $y_2$.   Thus, we conclude that (5.12) has no solutions when $a=2661$ and $b=20$.  

In practice, the standard bootstrapping technique can almost always achieve any desired lower bound on $x_2-x_1$, even when $B$ is as large as $10^{100}$; the extended idea to use $X_0>X_1$ and $Y_0 > Y_1$ allows the bootstrapping to succeed in the rare cases where the initial bootstrapping fails, but this extended idea requires $B$ small enough to do calculations for each $x_2 \le B$, so we are fortunate in our case that our bound $B$ is small (probably bounds up to $10^{10}$ would be doable).   
 
We give a further example of bootstrapping relevant to Section 4.  Let $(a,b) = (10,3)$ and rewrite $10^{x_1} + 3^{y_2} = 10^{x_2} + 3^{y_1}$ (with $x_1<x_2$) as $10^{x_1} (10^{x_2-x_1} - 1) = 3^{y_1} (3^{y_2-y_1} -1)$. First, $M(10,3)=4$ and $3^4 -1 = 2^4 \cdot 5$, and $M(3,10)=1$ and $10^1-1=3^2$. So $x_1 \ge 4 = X_0$ and $y_1 \ge 2 = Y_0$. Now $M(10^4,3) = 500$, so $500 \mid y_2-y_1$.  We use the Maple to find some medium-sized factors of $3^{500} - 1$, for instance, the prime $61070817601$. Note that $M(61070817601, 10) = 15267704400 = 2^4 \cdot 5^2 \cdot 3^2 \cdot 13 \cdot 139 \cdot 2347$.  Any divisor of 15267704400 divides $x_2-x_1$. We choose the divisor $9 \cdot 13$, and consider factors of $3^{9 \cdot 13} -1$; we find that $1846794457 \mid 3^{9 \cdot 13} -1$.  Now $M(1846794457,10) = 205199384$ so $205199384 \mid x_2-x_1$, showing that $x_2>205199384>1194836$.  We can continue (essentially reversing the roles of $a$ and $b$) by noting that $2^3 \cdot 13 \mid 205199384$, that $1846794457 \mid 3^{2^3 \cdot 13} -1$, and that $M(1846794457,3) = 205199384> 1194836$, hence $y_2>y_2-y_1>1194836$.  Thus, when either $(a,b) = (10,3)$ or $(3,10)$, there cannot be a second solution to (1.1).

\section*{Acknowledgments} 

This work used the Augie High Performance Computing cluster, funded by award NSF 2018933, at Villanova University, as well as Sage \cite{Sage}. The bootstrapping computations were performed using Maple\textsuperscript{TM}.   We are grateful to Aaron Wemhoff at Villanova and to William Stein at SageMath for their helpful assistance.  We are also grateful to Takafumi Miyazaki for many helpful comments, to Jeffrey Goodwin for helpful proof-readings, and to Reese Scott for many helpful suggestions and for providing the proofs of Theorem 2, Lemma 5, Propositions 1--3, and Lemma 6.

\end{document}